\documentclass[12pt,a4paper]{article}
\usepackage{graphicx} 

\usepackage{amsfonts, amsmath, amssymb, amsthm}
\usepackage{tabularx}
\usepackage{xcolor}
\usepackage{rotating}

\usepackage{indentfirst}
\usepackage[colorlinks=true, linkcolor=blue, citecolor=blue, urlcolor=blue]{hyperref}

\usepackage{multirow}
\usepackage{array}
\usepackage{booktabs}
\usepackage{float}
\usepackage{caption}

\theoremstyle{plain}
\newtheorem{remark}{Remark}

\newtheorem{theorem}{Theorem}[section]
\newtheorem{lemma}{Lemma}[section]
\newtheorem{result}{Result}[section]

\theoremstyle{plain}

\theoremstyle{definition}

\usepackage{natbib}
\setcitestyle{authoryear, open={(},close={)}}

\begin{document}

\title{\textbf{\Large To Study Properties of a Known Procedure in Adaptive Sequential Sampling Design}}

\author{Sampurna Kundu, Jayant Jha and Subir Kumar Bhandari\\
\small{Interdisciplinary Statistical Research Unit, Indian Statistical Institute, Kolkata, India}\\
\small{ \textit{\href{sampurna.kundu58@gmail.com}{sampurna.kundu58@gmail.com}, 
       \href{jayantjha@gmail.com}
         {jayantjha@gmail.com}, \href{subirkumar.bhandari@gmail.com}{subirkumar.bhandari@gmail.com}}}   
}

\date{}
\maketitle

\begin{abstract}
We consider the procedure proposed by Bhandari et al. \citeyearpar{Bhandari2009} in the context of two-treatment clinical trials, with the objective of minimizing the applications of the less effective drug to the least number of patients. Our focus is on an adaptive sequential procedure that is both simple and intuitive. Through a refined theoretical analysis, we establish that the number of applications of the less effective drug is a finite random variable whose all moments are also finite. In contrast, Bhandari et al. \citeyearpar{Bhandari2009} observed that this number increases logarithmically with the total sample size. We attribute this discrepancy to differences in their choice of the initial sample size and the method of analysis employed. We further extend the allocation rule to multi-treatment setup and derive analogous finiteness results, reinforcing the generalizability of our findings. Extensive simulation studies and real-data analyses support theoretical developments, showing stabilization in allocation and reduced patient exposure to inferior treatments as the total sample size grows. These results enhance the long-term ethical strength of the proposed adaptive allocation strategy.
\end{abstract}

\setlength{\parindent}{0pt}

\section{Introduction}
\setlength{\parindent}{20pt}

The field of adaptive sequential design has a long history of research, particularly in optimizing multiple objective functions. Berry and Fristedt \citeyearpar{Berry85} in their book on Bandit Problems, explored Bayesian methods to achieve optimal designs in this context. There has been substantial research in the field of adaptive sequential analysis (see e.g., Friedman et al. \citeyearpar{Friedman2010}; Ivanova et al. \citeyearpar{Ivanova2000}; Rosenberger et al. \citeyearpar{Rosenberger2004}).

In conventional sequential analysis for the two-population problem, samples are drawn one by one at each stage from both populations, ensuring that the sample sizes remain equal. The primary objectives in these cases are to make inferences about the population parameters with less error while minimizing the average sample number.

In contrast, adaptive sampling allows for unequal sampling from the two populations at each stage (sometimes selecting a sample exactly from one population). The choice of which population to sample at each stage depends on the performance of previously collected data. This approach has a natural connection with the context of most studied clinical trials, where two treatments are applied sequentially to a series of patients. The goal in such trials is to balance the need for statistical accuracy with ethical considerations, specifically applying the less effective treatment to a minimum number of patients.

Modern methodologies align well with the objectives of adaptive sequential sampling by developing advanced designs across various settings, such as covariate adjustment, ordinal responses, crossover trials and multi-treatment response adaptive design --- that primarily emphasize allocation strategies (or rules) and efficiency rather than bounding the usage of suboptimal treatments (Das et al. \citeyearpar{Das2023}; Bandyopadhyay et al. \citeyearpar{Bandyopadhyay2020}; Biswas et al. \citeyearpar{Biswas2020}; Biswas et al. \citeyearpar{PhaseIII2008}, pp. $33$ -- $53$). Das \citeyearpar{DasR2024} and Das \citeyearpar{DasS2024} devoted much effort to designing ethically strong procedures with reduction in allocations to the less effective treatment, including considerations for misclassifications and adaptive interim decisions. However, our approach directly focuses on this ethical aspect by ensuring finiteness of the expected number of applications of the less effective treatment.

In this context, Bhandari et al. \citeyearpar{Bhandari2007} considered the case with known face value of the parameters and concluded that the expected number of applications of the less effective drug is finite under an adaptive sequential design. However, where the parameters are totally unknown, Bhandari et al. \citeyearpar{Bhandari2009} found that the expected number of applications of the less effective drug increases logarithmically with the total sample size. This result was derived under the assumption of a large initial sample size.

In the present work, we modify the procedure proposed in Bhandari et al. \citeyearpar{Bhandari2009} and provide a more refined analysis to elucidate that the expected number of applications of the less effective drug remains finite. In fact, we show that all the moments of the random variable that denotes the number of applications of the less effective drug, are finite.

The paper is structured as follows. We explicate the two-treatment procedure in Section \ref{Section 2}, followed by a theorem with some remarks in Section \ref{Section 3}. The proof of the theorem is provided in Section \ref{Section 4}. In Section \ref{Section 5}, we introduce the extended version of the two-treatment procedure in the context of multi-treatment procedure and present its theoretical properties. Section \ref{Section 6} assesses the results of comprehensive simulation studies under various parameter choices. Section \ref{Section 7} illustrates the effectiveness of the proposed methods through real-data analyses based on two clinical trial datasets. Finally, Section \ref{Section 8} concludes the paper.

\section{Preliminaries} \label{Section 2}
\setlength{\parindent}{0pt}

Let $\mathcal{X} = X_{1}, X_{2}, X_{3}, \ldots$ follow the i.i.d.
$\operatorname{N}(\theta_{0}, \sigma_{0}^{2})$ and $\mathcal{Y} = Y_{1}, Y_{2}, Y_{3}, \ldots$ follow the i.i.d. $\operatorname{N}(\theta_{1}, \sigma_{1}^{2})$ and they be two independent data streams. We will use population $0$ and population $1$ for $\mathcal{X}$ and $\mathcal{Y}$, respectively. We draw samples adaptively, i.e., we draw samples sequentially and at stage $n$, after drawing a total of $n$ samples, we define the following two count variables:
\begin{align*}
N^\prime_{0, n} := \text{ number of samples drawn from } \mathcal{X},\\
N^\prime_{1, n} := \text{ number of samples drawn from } \mathcal{Y}.
\end{align*}

Let $\hat{\theta}_{0, n}$ and $\hat{\theta}_{1, n}$ denote the sample means at stage $n$ for $\mathcal{X}$ and $\mathcal{Y}$, respectively.
We adopt the following allocation rule:

\begin{enumerate}

\item[(i)] If $\hat{\theta}_{0, n} - \hat{\theta}_{1, n} > 0$, we increase $N^\prime_{0, n}$ by $1$,

\item[(ii)] If $\hat{\theta}_{0, n} - \hat{\theta}_{1, n} < 0$, we increase $N^\prime_{1, n}$ by $1$,

\item[(iii)] If $\hat{\theta}_{0, n} - \hat{\theta}_{1, n} = 0$, we increase either $N^\prime_{0,n}$ or $N^\prime_{1, n}$ by $1$ with probability $\frac{1}{2}$ each.

\end{enumerate}

There are $N$ patients in total. Finally, when $n=N$, we accept the null $H_{0}: \theta_{0} > \theta_{1}$, with probability $1$, if $\hat{\theta}_{0, N} - \hat{\theta}_{1, N} > 0$ and with probability $\frac{1}{2}$, if $\hat{\theta}_{0, N} - \hat{\theta}_{1, N} = 0$.

\begin{remark}
    The procedure described above was extensively studied by Bhandari et al. \citeyearpar{Bhandari2009}, where it was referred to as \emph{Procedure III}.
\end{remark}

\begin{remark}
    In the proof of the main result, we assume $\sigma_{0}^{2}$ and $\sigma_{1}^{2}$ to be unity and $\theta_{0}>\theta_{1}$ without any loss of generality.
\end{remark}

\section{Main Result} \label{Section 3}

Let us define the following:
$$N_{1, n} := \min \{N^\prime_{1, n}, N^\prime_{0, n}\}\quad \text{and} \quad N_{1, N} := \min \{N^\prime_{1, N}, N^\prime_{0, N}\}.$$ 

\begin{theorem} \label{Theorem 3.1}
Under the notations and assumptions given in the Section \ref{Section 2}, we have the following theorem:
\begin{enumerate}

\item[(i)] $\lim\limits_{N \to \infty} N_{1, N} < \infty$, i.e., $N_{1, N}$ is a finite random variable.

\item[(ii)] All the moments of $N_{1, N}$ are bounded (uniformly over $N$).

\end{enumerate}
\end{theorem}

\begin{remark}
    Bhandari et al. \citeyearpar{Bhandari2009} found that both $\frac{N^\prime_{1, N}}{\log (N)}$ and $\frac{\mathbb{E}(N^\prime_{1, N})}{\log (N)}$ tend to positive constants, contrary to our main result. This variation arises because they started with a large initial sample size and employed a different method of (mathematical) analysis.
\end{remark}

\begin{remark}
    Note that, $\lim\limits_{N \to \infty} N_{1, N} \neq \lim\limits_{N \to \infty} N^\prime_{1, N}$ if and only if there is incorrect inference about $H_{0}: \theta_{0} > \theta_{1}$. The probability of incorrect inference is negligible and it is less than or equal to 
\[
\Phi\left(-\left(\theta_{0} - \theta_{1} + \varepsilon\right) \cdot \sqrt{N_{1, N}}\right) < \Phi\left(-\left(\theta_{0} - \theta_{1} + \varepsilon\right) \cdot \sqrt{M}\right),
\]
for some $\varepsilon > 0$, where $M$ is a finite natural number and $\Phi(\cdot)$ is the CDF of the standard normal distribution. We start with an initial sample size of $M$ from each population at the beginning of the procedure.
\end{remark}

\begin{remark}
Note that, in the case of sequential two-population testing problem, when both $\theta_{0}$ and $\theta_{1}$ are unknown, it is not possible to achieve consistent hypothesis testing (i.e., PCS approaching 1 as the total sample size tends to infinity) if the expected number of samples from one population remains finite while the other grows indefinitely.
\end{remark}
 
\section{Proof of Main Result} \label{Section 4}

The following two results, as stated in Chapter $1$, Section $9$ of Billingsley \citeyearpar{Billingsley}, are presented here without proof.

\begin{result} \label{Result 4.1}
(Law of Iterated Logarithm)
Let $\left\{X_{n}\right\}$ be i.i.d. random variables with mean zero and unit variance. Let $S_{n}=X_{1}+X_{2}+X_{3}+\ldots+X_{n}$.
Then, $$\limsup\limits_{n \to \infty} \frac{\left|S_{n}\right|}{\sqrt{2 n \log \operatorname{logn}}} = 1 \quad \text{almost surely.}$$
\end{result}

\begin{result} \label{Result 4.2}
(Chernoff's Theorem in the context of Normal Distribution)
Let $\left\{X_{n}\right\}$ be the i.i.d. random variables following $\operatorname{N}(\mu, 1)$ with $\mu < 0$, with mgf $M(t)$. Suppose $\bar{X}_n=\frac{1}{n}\left(X_{1}+X_{2}+X_{3}+\ldots+X_{n}\right)$. If $\rho = \min\limits_t M(t)$, then $\rho < 1$ and
$$\lim\limits_{n \to \infty} \frac{1}{n} \log \left[\mathbb{P}\left(\bar{X}_n \geq 0\right)\right]=\log (\rho),$$
which implies $\mathbb{P}\left(\bar{X}_n > 0\right) < \mathrm{C} \cdot \rho^{n} \hspace{2mm} \forall \hspace{2mm} n$, for some positive constant $\mathrm{C}$.
\end{result}

\begin{remark}
Note that, along with normal distribution for $\{X_n\}$, Chernoff's theorem also holds for any distribution with $\mathbb{E}(X_n)<0$ and $\mathbb{P}(X_n >0)>0$ (see, Billingsley \citeyearpar{Billingsley}).
\end{remark}

Using Result \ref{Result 4.1}, we find that for any given $\varepsilon>0$, there exists a natural number $N_{\varepsilon}^{*}$ and subset $A$ of sample space with $\mathbb{P}(A) > 1 - \varepsilon$, such that
$$\forall \hspace{2mm} \omega \in A, \quad |\hat{\theta}_{0, (n)}(\omega)-\theta_{0}| < \varepsilon \text{ and } |\hat{\theta}_{1, (n)}(\omega)-\theta_{1}| < \varepsilon \text{, (for all }n > N_{\varepsilon}^{*}\text{),}$$
where $\hat{\theta}_{0, (n)}$ and $\hat{\theta}_{1, (n)}$ are the mean of first $n$ observations coming from the data streams $\mathcal{X}$ and $\mathcal{Y}$  respectively.

Using the above, we find that with probability $\geq(1-\varepsilon)$,
$$
\lim\limits_{N \to \infty} N_{1, N}<N_{\varepsilon}^{*} \text{, i.e., } N_{1, N} \text{ is a finite random variable.}
$$
This proves part (i) of \autoref{Theorem 3.1}.
\\
Let us define
\begin{equation} \label{(1)}
    M_{u}^{*} = \inf \left\{ k : \sup\limits_{n \geq k} \bar{Z}_n \leq u \right\},
\end{equation}

where $Z_{1}, Z_{2}, Z_{3}, \ldots, Z_{n}$ are i.i.d. $\operatorname{N}(0,1)$ and $\bar{Z}_n = \frac{1}{n} \sum_{i=1}^{n} Z_{i}$.\\
Define $X_{i}=Z_{i}-u$ for small $u > 0$, for every natural number $i$.\\
Then, Result \ref{Result 4.2} holds for this sequence of $X_{i}$.

\begin{remark}
Let $u=\frac{\theta_0-\theta_1}{3}$. Then, $N_{1,n} \leq M_{u0}^*+M_{u1}^*$, for some $M_{ui}^*$ ($i=0,1$) which are i.i.d. $M_u^*$. This is because there exists $M_{u0}^*$ and $M_{u1}^*$ (independent) such that $\hat{\theta}_{0, (n)} \in (\theta_0-u, \theta_0 +u)$ for all $n \geq M_{u0}^*$, $\hat{\theta}_{1, (n)} \in (\theta_1-u, \theta_1 +u)$ for all $n \geq M_{u1}^*$ and $N_{1,n}=min\{N_{1,n}',N_{0,n}'\} \leq N_{0,n}'+N_{1,n}' \leq M_{u0}^*+M_{u1}^*$.
\end{remark}
The proof of part (ii) of \autoref{Theorem 3.1} will be complete upon proving Lemma \ref{Lemma 4.1}.

\begin{lemma} \label{Lemma 4.1}
    The mgf of the random variable $M_{u}^{*}$ is finite in a small open neighbourhood of 0 i.e., for  $t \in (-\delta, \delta)$ where $\delta >0$ and $\delta$ is small, $\mathbb{E}(\exp(t M_u^*)) < \infty$. Hence, all the moments of the random variable $M_u^*$ are finite.
\end{lemma}

\begin{proof}
    We will prove the finiteness of the first moment only, i.e.,  $\mathbb{E}\left(M_{u}^{*}\right)<\infty$, as the proof for the mgf is analogous.
\begin{align*}
\mathbb{P}\left(M_{u}^{*} > k\right) & = \mathbb{P}\left[\sup\limits_{n \geq k} \bar{Z}_n > u \right] \\ \notag
& \leq \mathbb{P}\left[\bigcup\limits_{n \geq k}\left(\bar{Z}_n \geq u\right)\right] \\\notag
& \leq \sum_{n=k}^{\infty} \mathbb{P}\left(\bar{X}_n \geq 0\right)\left[\text{since } \bar{X}_n = \bar{Z}_n - u\right] \\\notag
& \leq \sum_{n=k}^{\infty} C\cdot \rho^{n} \quad [\text{using Result } \ref{Result 4.2} ] \\\notag
& = \frac{C \cdot \rho^{k}}{1 - \rho} \quad[\text{this holds for large } k]
\end{align*}

Hence, for large $k$,
$\mathbb{E}\left(M_{u}^{*}\right) \leq k+\frac{C}{(1-\rho)^{2}} \rho^{k}<\infty$.
\end{proof}

\section{Extension to Multiple Treatments} \label{Section 5}
\setlength{\parindent}{0pt}

\subsection{Multi-Treatment Allocation Rule} \label{Section 5.1}

Suppose that we have $m$ ($\geq 2$) competing drugs. Consider $m$ populations: $\mathcal{X}_{1}$, $\mathcal{X}_{2}$, \ldots, $\mathcal{X}_{m}$. The random variables in population $\mathcal{X}_{j}$ ($1 \leq j \leq m$) have i.i.d. $\operatorname{N}(\theta_{j}, \sigma_{j}^{2})$ distribution. We draw samples adaptively, i.e., we draw samples sequentially and at stage $n$, after drawing a total of $n$ samples, we define the following count variables:
\begin{align*}
\text{for } 1 \leq j \leq m\text{, } N^\prime_{j, n} := \text{ number of samples drawn from } \mathcal{X}_{j} \text{.}
\end{align*}
Let $\hat{\theta}_{n, j}$ denotes the sample mean at stage $n$ for $\mathcal{X}_{j}$ for $1 \leq j \leq m$.
We adopt the following allocation rule:
\begin{enumerate}

\item[(i)] Increase $N^\prime_{j^{*}, n}$ by $1$, where $j^{*}$ = $\underset{1 \leq j \leq m}{\arg\max}$ $\hat{\theta}_{n, j}$,

\item[(ii)] If there are $s$ many argmaxes, then increase those count variables with probability $\frac{1}{s}$ each.

\end{enumerate}

\begin{remark}
    We assume $\theta_{1} > \theta_{2} > \ldots > \theta_{m}$ without any loss of generality and $\sigma_{j}^{2}$ to be unity for each $j$.
\end{remark}

\subsection{Theoretical Properties} \label{Section 5.2}

Let us define the following:
\begin{align*}
\hat{j_{n}} := \underset{j}{\arg\max} \hspace{1mm} N^\prime_{j, n}, \hspace{2mm} N_{1, n} := \max\limits_{\substack{1 \leq j \leq m \\ j \neq \hat{j}_{n}}} \hspace{1mm} N^\prime_{j, n},\\
\hat{j_{N}} := \underset{j}{\arg\max} \hspace{1mm} N^\prime_{j, N}, \hspace{2mm} \text{and} \hspace{2mm} N_{1, N} := \max\limits_{\substack{1 \leq j \leq m \\ j \neq \hat{j}_{N}}} \hspace{1mm} N^\prime_{j, N}.
\end{align*}

\begin{theorem} \label{Theorem 5.1}
Under the notations and assumptions given in the Section \ref{Section 5.1}, we have the following theorem:
\begin{enumerate}

\item[(i)] $\lim\limits_{N \to \infty} N_{1, N} < \infty$, i.e., $N_{1, N}$ is a finite random variable.

\item[(ii)] All the moments of $N_{1, N}$ are bounded (uniformly over $N$).

\end{enumerate}
\end{theorem}

\subsection{Proof of \autoref{Theorem 5.1}} \label{Section 5.3}

Using Result \ref{Result 4.1}, we find that for any given $\varepsilon>0$, there exists a natural number $N_{\varepsilon}^{*}$ and subset $A$ of sample space with $\mathbb{P}(A) > 1 - \varepsilon$, such that
$$\forall \hspace{2mm} \omega \in A, \quad \forall \hspace{2mm} n \geq N_{\varepsilon}^{*}, \quad \forall \hspace{2mm} j \in {1, \dots, m}, \quad |\hat{\theta}_{j, (n)}(\omega)-\theta_{j}| < \varepsilon,$$
where $\hat{\theta}_{j, (n)}(\omega)$ denotes the mean of first $n$ observations coming from the data streams $\mathcal{X}_{j}$.

Using the above, we find that with probability $\geq(1-\varepsilon)$,
$$
\lim\limits_{N \to \infty} N_{1, N}<N_{\varepsilon}^{*} \text{, i.e., } N_{1, N} \text{ is a finite random variable.}
$$
This proves part (i) of \autoref{Theorem 5.1}.

We define $M_{u}^{*}$ as in \autoref{(1)}.

\begin{remark}
Let $u=\frac{\min\limits_{1 \leq j \leq m-1} {(\theta_j-\theta_{j+1})}}{3}$. Then, for each $j$, there exists $M_{uj}^* (\overset{d}{\equiv} M_{u}^{*})$, such that $\forall$ $n \geq M_{uj}^*$, $\hat{\theta}_{j, (n)} \in (\theta_j-u, \theta_j+u)$.
\end{remark}

Then

$$N_{1, n} = \max\limits_{\substack{1 \leq j \leq m \\ j \neq \hat{j}_{n}}} \hspace{1mm} N^\prime_{j, n} \leq \sum\limits_{\substack{1 \leq j \leq m \\ j \neq \hat{j}_{n}}} \hspace{1mm} N^\prime_{j, n} \leq \sum\limits_{\substack{1 \leq j \leq m \\ j \neq \hat{j}_{n}}} \hspace{1mm} M_{uj}^* \leq \sum\limits_{1 \leq j \leq m} \hspace{1mm} M_{uj}^*.$$

Therefore, $N_{1, n} \leq \sum\limits_{1 \leq j \leq m} \hspace{1mm} M_{uj}^*$.

Each $M_{uj}^*$ has bounded moments (using Lemma \autoref{Lemma 4.1}). Also, $M_{uj_{1}}^*$ and $M_{uj_{2}}^*$ are independent $\forall \hspace{1mm} j_1 \neq j_2$, as the populations are independent. So, $N_{1,n}$ also has finite moments.

This proves part (ii) of \autoref{Theorem 5.1}.



\section{Simulation Study} \label{Section 6}
\setlength{\parindent}{20pt}

We conduct simulation studies to evaluate the performance of our adaptive sampling procedure as outlined in Section \ref{Section 2}. In each replication, we generate samples of $X_i$'s and $Y_j$'s adaptively according to the allocation rule and the estimators $\hat{\theta}_{0,n}$ and $\hat{\theta}_{1,n}$ are updated sequentially at each stage $n$. The procedure continues until the stopping condition is met at $n = N$, upon which we record the values of $N_{1,N}$ and whether a correct selection (CS) was made, indicated by a binary outcome ($1$ for correct, $0$ for incorrect).

This entire procedure is repeated $10,000$ times for each setting, and we compute the empirical average of $N_{1,N}$ and the proportion of CS to estimate $\mathbb{E}(N_{1,N})$ and the probability of correct selection (PCS), respectively.

To examine the limiting behavior of $N_{1,N}$ as the total sample size $N$ increases and to support the theoretical findings of Section \ref{Section 3}, we implement the allocation rule described in Section \ref{Section 2} for various parameter combinations $(\theta_0, \theta_1, \sigma_0, \sigma_1)$ in the context of Normal populations. For each $N$ (from moderate to large), we conduct $10,000$ replications and use the simulated data to compute PCS as the proportion of correct decisions. In addition, we estimate $\mathbb{E}(N_{1,N})$ for different parameter settings and analyze its limiting values as $N$ increases.

The simulation results for different mean pairs $(\theta_0, \theta_1)$ are summarized in \autoref{Table 1}. The analysis indicates that $\mathbb{E}(N_{1,N})$ stabilizes to a constant as $N$ increases. The allocation rule consistently assigns more samples to the population with a higher mean value. Each simulation starts with equal initial sample sizes for both populations, determined on the basis of the underlying parameters.

\begin{table*}[h]
\centering
\caption{Simulation results for two Normal populations with distinct mean pairs $(\theta_{0}, \theta_{1})$ and constant variance pair $(\sigma_{0}^{2}, \sigma_{1}^{2}) = (1, 0.7)$ under a fixed initial sample size for two-treatment procedure.} \label{Table 1}
\begin{tabular}{l cc cc cc}
\hline
Total Sample  & \multicolumn{2}{c}{$(\theta_0, \theta_1) = (0.5, 0)$} 
    & \multicolumn{2}{c}{$(\theta_0, \theta_1) = (0.8, 0.2)$} 
    & \multicolumn{2}{c}{$(\theta_0, \theta_1) = (1, 0.5)$} \\
\cline{2-7}
\hspace{2mm} Size $(N)$ & \multicolumn{1}{c}{PCS} & \multicolumn{1}{c}{$\mathbb{E}(N_{1, N})$} 
    & \multicolumn{1}{c}{PCS} & \multicolumn{1}{c}{$\mathbb{E}(N_{1, N})$} 
    & \multicolumn{1}{c}{PCS} & \multicolumn{1}{c}{$\mathbb{E}(N_{1, N})$} \\
\hline
\hspace{5mm} 200  & 0.9396  & 16.4209  & 0.9718  & 15.7146  & 0.9363  & 16.4576 \\
\hspace{5mm} 300  & 0.9407  & 16.7092  & 0.9710  & 15.8912  & 0.9377 & 16.7389 \\
\hspace{5mm} 400  & 0.9412  & 16.8754  & 0.9709  & 15.9557  & 0.9415 & 16.8625 \\
\hspace{5mm} 800  & 0.9413  & 17.6627  & 0.9726  & 16.0868  & 0.9430  & 17.4870 \\
\hspace{5mm} 900  & 0.9467  & 17.1830  & 0.9735  & 16.1228  & 0.9429  & 17.7404 \\
\hspace{3mm} 1000 & 0.9421  & 17.5419  & 0.9732  & 16.2150  & 0.9405  & 17.2427 \\
\hspace{3mm} 1500 & 0.9428  & 18.3047  & 0.9679  & 16.3760  & 0.9407  & 18.1658 \\
\hspace{3mm} 2000 & 0.9452  & 18.1094  & 0.9717  & 16.4335  & 0.9426  & 18.1526 \\
\hspace{3mm} 2500 & 0.9452  & 18.1166  & 0.9718  & 16.4929  & 0.9432  & 18.0036 \\
\hspace{3mm} 3000 & 0.9409  & 18.4554  & 0.9715  & 16.5518  & 0.9344  & 18.3848 \\
\hspace{3mm} 3500 & 0.9423  & 18.5810  & 0.9722  & 16.5738  & 0.9376  & 18.9482 \\
\hline
\end{tabular}
\end{table*}

Chernoff's theorem also holds for other distributions. To demonstrate the broader applicability of our approach, we conduct a similar simulation study for Bernoulli populations. The simulation results for three specific parameter configurations of $(p_0, p_1)$, where $p_i$ denotes the success probability for population $i = 0, 1$, are presented in \autoref{Table 2}. As in the case of Normal distributions, the findings indicate that $\mathbb{E}(N_{1,N})$ converges to a finite value as $N$ increase. Each scenario begins with equal initial sample sizes for both populations to ensure a balanced start of the procedure.

\begin{table*}[h]
\centering
\caption{Simulation results for two Bernoulli populations with distinct success probability pairs $(p_0, p_1)$ under a fixed initial sample size for two-treatment procedure.} \label{Table 2}
\begin{tabular}{l cc cc cc}
\hline
Total Sample  & \multicolumn{2}{c}{$(p_0, p_1) = (0.5, 0.2)$} 
    & \multicolumn{2}{c}{$(p_0, p_1) = (0.6, 0.3)$} 
    & \multicolumn{2}{c}{$(p_0, p_1) = (0.8, 0.5)$} \\
\cline{2-7}
\hspace{2mm} Size $(N)$ & \multicolumn{1}{c}{PCS} & \multicolumn{1}{c}{$\mathbb{E}(N_{1, N})$} 
    & \multicolumn{1}{c}{PCS} & \multicolumn{1}{c}{$\mathbb{E}(N_{1, N})$} 
    & \multicolumn{1}{c}{PCS} & \multicolumn{1}{c}{$\mathbb{E}(N_{1, N})$} \\
\hline
\hspace{5mm} 200  & 0.9714  & 15.7156  & 0.9613 & 15.8792  & 0.9700  & 15.6926 \\
\hspace{5mm} 300  & 0.9704  & 15.8346  & 0.9623 & 15.9476  & 0.9667 & 15.7851 \\
\hspace{5mm} 400  & 0.9678  & 15.9646  & 0.9626  & 16.0014  & 0.9721 & 15.7609 \\
\hspace{5mm} 800  & 0.9690   & 16.0128  & 0.9648  & 16.5692  & 0.9732 & 15.9030 \\
\hspace{5mm} 900  & 0.9710   & 16.1097  & 0.9649  & 16.4865  & 0.9701 & 15.9566 \\
\hspace{3mm} 1000 & 0.9676  & 16.3659  & 0.9617  & 16.4393  & 0.9722 & 16.1128 \\
\hspace{3mm} 1500 & 0.9684  & 16.3968  & 0.9648  & 16.6400  & 0.9710  & 16.1403 \\
\hspace{3mm} 2000 & 0.9727  & 16.2998  & 0.9642  & 16.5592  & 0.9709 & 16.3007 \\
\hspace{3mm} 2500 & 0.9731  & 16.3938  & 0.9661  & 16.5206  & 0.9714 & 16.4970 \\
\hspace{3mm} 3000 & 0.9676  & 17.0703  & 0.9661  & 16.6232  & 0.9698 & 16.9098 \\
\hspace{3mm} 3500 & 0.9709  & 16.6486  & 0.9673  & 16.3406  & 0.9671 & 16.3561 \\
\hline
\end{tabular}
\end{table*}


In addition, we analyze cases where both populations are identically distributed. \autoref{Table 3} reports simulation results for this setting under both Normal and Bernoulli populations. These represent scenarios in which the treatments are equally effective. Accordingly, PCS is approximately $0.5$, reflecting the indistinguishability of the two treatments. Notably, $\mathbb{E}(N_{1,N})$ increases slowly but remains proportionally smaller relative to $N$, underscoring the bounded nature of the allocation count even under treatment equivalence.

\begin{table*}[h!]
\centering
\caption{Simulation results based on two identically distributed populations under a fixed initial sample size for two-treatment procedure.} \label{Table 3}
\resizebox{\textwidth}{!}{%
\begin{tabular}{c c c c c}
\hline
\multicolumn{5}{c}{$\operatorname{N}(1, 1)$ populations} \\
\hline
Total Sample Size $(N)$ 
& \multicolumn{1}{c}{PCS} 
& \multicolumn{1}{c}{$\mathbb{E}(N_{1, N})$} 
& \multicolumn{1}{c}{$\min \{\mathbb{E}(N^\prime_{0, N}), \mathbb{E}(N^\prime_{1, N})\}$} 
& \multicolumn{1}{c}{$\min \{\mathbb{E}(N^\prime_{0, N}), \mathbb{E}(N^\prime_{1, N})\}/N$} \\
\hline
\hspace{1mm} 200  & 0.4987  & \hspace{1mm} 56.5952    & \hspace{2mm} 99.8838   & 0.499419 \\
\hspace{1mm} 300  & 0.4994  & \hspace{1mm} 61.9217   & \hspace{1mm} 149.6241  & 0.498747 \\
\hspace{1mm} 400  & 0.5098  & \hspace{1mm} 65.9665   & \hspace{1mm} 197.1139  & 0.492785 \\
\hspace{1mm} 800  & 0.5084  & \hspace{1mm} 76.9530   & \hspace{1mm} 394.5546  & 0.493193 \\
\hspace{1mm} 900  & 0.4994  & \hspace{1mm} 79.9105   & \hspace{1mm} 449.5351  & 0.499483 \\
1000 & 0.5015  & \hspace{1mm} 79.7329   & \hspace{1mm} 498.0089  & 0.498009 \\
1500 & 0.4949  & \hspace{1mm} 88.1929   & \hspace{1mm} 743.1995  & 0.495466 \\
2000 & 0.4962  & \hspace{1mm} 91.7086   & \hspace{1mm} 993.3254  & 0.496663 \\
2500 & 0.4939  & \hspace{1mm} 96.0888   & 1237.8096 & 0.495124 \\
3000 & 0.5069  & \hspace{1mm} 97.8849   & 1477.5369 & 0.492512 \\
3500 & 0.5009  & 101.4421  & 1746.3167 & 0.498948 \\
\hline
\multicolumn{5}{c}{Bernoulli($0.5$) populations} \\
\hline
Total Sample Size $(N)$ 
& \multicolumn{1}{c}{PCS} 
& \multicolumn{1}{c}{$\mathbb{E}(N_{1, N})$} 
& \multicolumn{1}{c}{$\min \{\mathbb{E}(N^\prime_{0, N}), \mathbb{E}(N^\prime_{1, N})\}$} 
& \multicolumn{1}{c}{$\min \{\mathbb{E}(N^\prime_{0, N}), \mathbb{E}(N^\prime_{1, N})\}/N$} \\
\hline
\hspace{1mm} 200  & 0.5047  & \hspace{1mm} 56.1546   & \hspace{2mm} 99.5748     & 0.497874 \\
\hspace{1mm} 300  & 0.5014  & \hspace{1mm} 61.7647   & \hspace{1mm} 149.9355     & 0.499785 \\
\hspace{1mm} 400  & 0.5054  & \hspace{1mm} 65.7200   & \hspace{1mm} 198.4916     & 0.496229 \\
\hspace{1mm} 800  & 0.5016  & \hspace{1mm} 77.5626   & \hspace{1mm} 398.6040     & 0.498255 \\
\hspace{1mm} 900  & 0.5060  & \hspace{1mm} 79.6830   & \hspace{1mm} 445.1634     & 0.494626 \\
1000 & 0.5060  & \hspace{1mm} 79.5445   & \hspace{1mm} 494.7277     & 0.494728 \\
1500 & 0.5007  & \hspace{1mm} 87.8227   & \hspace{1mm} 749.0695     & 0.499380 \\
2000 & 0.5048  & \hspace{1mm} 90.6309   & \hspace{1mm} 992.2599     & 0.496130 \\
2500 & 0.5003  & \hspace{1mm} 93.9105   & 1247.3643     & 0.498946 \\
3000 & 0.4989  & \hspace{1mm} 99.5797   & 1497.4663     & 0.499155 \\
3500 & 0.5073  & 103.4383  & 1726.3047     & 0.493230 \\
\hline
\end{tabular}%
}
\end{table*}

\begin{table*}[h]
\centering
\caption{Simulation results for Normal populations with distinct mean triplets $(\theta_{1}, \theta_{2}, \theta_{3})$ and constant variance triplet $(\sigma_{1}^{2}, \sigma_{2}^{2}, \sigma_{3}^{2}) = (1, 0.7, 0.5)$ under a fixed initial sample size for multi-treatment procedure.} \label{Table 4}
\resizebox{\textwidth}{!}{%
\begin{tabular}{l | cc | cc}
\hline
Total Sample  & \multicolumn{2}{c|}{$(\theta_1, \theta_2, \theta_3) = (0.9, 0.2, 0)$} 
    & \multicolumn{2}{c}{$(\theta_1, \theta_2, \theta_3) = (2, 1.2, 0.5)$} \\
\cline{2-5}
\hspace{2mm} Size $(N)$
& PCS & $\mathbb{E}(\text{2nd max} \{N^\prime_{1, N}, N^\prime_{2, N}, N^\prime_{3, N}\})$
& PCS & $\mathbb{E}(\text{2nd max} \{N^\prime_{1, N}, N^\prime_{2, N}, N^\prime_{3, N}\})$ \\
\hline
\hspace{5mm} 200  & 0.9475  & \hspace{1mm} 9.3374  & 0.9527  & 6.9510 \\
\hspace{5mm} 300  & 0.9459  & \hspace{1mm} 9.6494  & 0.9503  & 7.0553 \\
\hspace{5mm} 400  & 0.9492  & 10.0209 & 0.9553  & 7.0896 \\
\hspace{5mm} 800  & 0.9425  & 10.8458 & 0.9503  & 7.9133 \\
\hspace{5mm} 900  & 0.9442  & 11.0578 & 0.9550  & 7.4772 \\
\hspace{3mm} 1000 & 0.9455  & 11.2956 & 0.9577  & 7.7322 \\
\hspace{3mm} 1500 & 0.9504  & 11.7424 & 0.9515  & 8.0190 \\
\hspace{3mm} 2000 & 0.9444  & 12.3375 & 0.9561  & 8.1553 \\
\hline
\end{tabular}%
}
\end{table*}

\begin{table*}[h]
\centering
\caption{Simulation results for two Normal populations with distinct mean pairs $(\theta_{0}, \theta_{1})$ and constant variance pair $(\sigma_{0}^{2}, \sigma_{1}^{2}) = (1, 0.7)$ based on \emph{Procedure III} of Bhandari et al. \citeyearpar{Bhandari2009} under a fixed initial sample size.} \label{Table 5}
\resizebox{\textwidth}{!}{%
\begin{tabular}{l | ccc | ccc | ccc}
\hline
Total Sample  & \multicolumn{3}{c|}{$(\theta_0, \theta_1) = (0.5, 0)$} 
    & \multicolumn{3}{c|}{$(\theta_0, \theta_1) = (0.8, 0.2)$} 
    & \multicolumn{3}{c}{$(\theta_0, \theta_1) = (1, 0.5)$} \\
\cline{2-10}
\hspace{2mm} Size $(N)$ 
& PCS & $\mathbb{E}(N^\prime_{1, N})$ & $\mathbb{E}(N^\prime_{1, N})/\log (N)$ 
& PCS & $\mathbb{E}(N^\prime_{1, N})$ & $\mathbb{E}(N^\prime_{1, N})/\log (N)$  
& PCS & $\mathbb{E}(N^\prime_{1, N})$ & $\mathbb{E}(N^\prime_{1, N})/\log (N)$ \\
\hline
\hspace{5mm} 200  & 0.9436  & \hspace{1mm} 24.6724  & \hspace{1mm} 4.6566  & 0.9725  & \hspace{1mm} 19.8617  & 3.7487  & 0.9442  & \hspace{1mm} 24.7791  & \hspace{1mm} 4.6768 \\
\hspace{5mm} 300  & 0.9408  & \hspace{1mm} 31.3563  & \hspace{1mm} 5.4975  & 0.9725  & \hspace{1mm} 22.6578  & 3.9724  & 0.9455  & \hspace{1mm} 29.9178  & \hspace{1mm} 5.2453 \\
\hspace{5mm} 400  & 0.9404  & \hspace{1mm} 37.3903  & \hspace{1mm} 6.2406  & 0.9699  & \hspace{1mm} 26.3717  & 4.4015  & 0.9447  & \hspace{1mm} 35.9691  & \hspace{1mm} 6.0034 \\
\hspace{5mm} 800  & 0.9401  & \hspace{1mm} 61.8690  & \hspace{1mm} 9.2554  & 0.9756  & \hspace{1mm} 34.0681  & 5.0965  & 0.9415  & \hspace{1mm} 60.8365  & \hspace{1mm} 9.1010 \\
\hspace{5mm} 900  & 0.9393  & \hspace{1mm} 68.3112  & 10.0422 & 0.9728  & \hspace{1mm} 39.2264  & 5.7666  & 0.9432  & \hspace{1mm} 65.5061  & \hspace{1mm} 9.6299 \\
\hspace{3mm} 1000 & 0.9372  & \hspace{1mm} 76.6053  & 11.0898 & 0.9692  & \hspace{1mm} 45.3261  & 6.5616  & 0.9410  & \hspace{1mm} 73.0315  & 10.5724 \\
\hspace{3mm} 1500 & 0.9377  & 107.6991 & 14.7266 & 0.9725  & \hspace{1mm} 55.8550  & 7.6375  & 0.9405  & 103.3666 & 14.1342 \\
\hspace{3mm} 2000 & 0.9411  & 131.7802 & 17.3374 & 0.9738  & \hspace{1mm} 67.7209  & 8.9096  & 0.9393  & 135.5299 & 17.8308 \\
\hspace{3mm} 2500 & 0.9444  & 153.7988 & 19.6572 & 0.9697  & \hspace{1mm} 91.2159  & 11.6584 & 0.9354  & 176.4448 & 22.5516 \\
\hspace{3mm} 3000 & 0.9429  & 185.7885 & 23.2051 & 0.9733  & \hspace{1mm} 95.9255  & 11.9812 & 0.9442  & 181.9366 & 22.7240 \\
\hspace{3mm} 3500 & 0.9415  & 218.7154 & 26.8017 & 0.9733  & 108.5402 & 13.3007 & 0.9426  & 214.6163 & 26.2993 \\
\hline
\end{tabular}%
}
\end{table*}

\begin{table*}[h]
\centering
\caption{Simulation results for two Bernoulli populations with distinct success probability pairs $(p_0, p_1)$ based on \emph{Procedure III} of Bhandari et al. \citeyearpar{Bhandari2009} under a fixed initial sample size.} \label{Table 6}
\resizebox{\textwidth}{!}{%
\begin{tabular}{l | ccc | ccc | ccc}
\hline
Total Sample  & \multicolumn{3}{c|}{$(p_0, p_1) = (0.5, 0.2)$} 
    & \multicolumn{3}{c|}{$(p_0, p_1) = (0.6, 0.3)$} 
    & \multicolumn{3}{c}{$(p_0, p_1) = (0.8, 0.5)$} \\
\cline{2-10}
\hspace{2mm} Size $(N)$ 
& PCS & $\mathbb{E}(N^\prime_{1, N})$ & $\mathbb{E}(N^\prime_{1, N})/\log (N)$ 
& PCS & $\mathbb{E}(N^\prime_{1, N})$ & $\mathbb{E}(N^\prime_{1, N})/\log (N)$  
& PCS & $\mathbb{E}(N^\prime_{1, N})$ & $\mathbb{E}(N^\prime_{1, N})/\log (N)$ \\
\hline
\hspace{5mm} 200  & 0.9707  & \hspace{1mm} 20.0885  & \hspace{1mm} 3.7915  & 0.9644  & \hspace{1mm} 21.0589  & \hspace{1mm} 3.9746  & 0.9693  & \hspace{1mm} 20.1448  & \hspace{1mm} 3.8021 \\
\hspace{5mm} 300  & 0.9694  & \hspace{1mm} 23.3904  & \hspace{1mm} 4.1009  & 0.9667  & \hspace{1mm} 24.0227  & \hspace{1mm} 4.2117  & 0.9707  & \hspace{1mm} 22.9941  & \hspace{1mm} 4.0314 \\
\hspace{5mm} 400  & 0.9686  & \hspace{1mm} 26.7283  & \hspace{1mm} 4.4611  & 0.9655  & \hspace{1mm} 27.8444  & \hspace{1mm} 4.6473  & 0.9720  & \hspace{1mm} 25.2552  & \hspace{1mm} 4.2152 \\
\hspace{5mm} 800  & 0.9722  & \hspace{1mm} 36.5467  & \hspace{1mm} 5.4673  & 0.9657  & \hspace{1mm} 41.2123  & \hspace{1mm} 6.1652  & 0.9723  & \hspace{1mm} 36.5258  & \hspace{1mm} 5.4642 \\
\hspace{5mm} 900  & 0.9695  & \hspace{1mm} 41.6838  & \hspace{1mm} 6.1278  & 0.9670  & \hspace{1mm} 43.9721  & \hspace{1mm} 6.4642  & 0.9710  & \hspace{1mm} 40.0918  & \hspace{1mm} 5.8938 \\
\hspace{3mm} 1000 & 0.9712  & \hspace{1mm} 43.0519  & \hspace{1mm} 6.2324  & 0.9644  & \hspace{1mm} 49.4955  & \hspace{1mm} 7.1652  & 0.9686  & \hspace{1mm} 45.1317  & \hspace{1mm} 6.5335 \\
\hspace{3mm} 1500 & 0.9691  & \hspace{1mm} 60.6695  & \hspace{1mm} 8.2959  & 0.9655  & \hspace{1mm} 65.7791  & \hspace{1mm} 8.9945  & 0.9699  & \hspace{1mm} 59.2646  & \hspace{1mm} 8.1038 \\
\hspace{3mm} 2000 & 0.9700  & \hspace{1mm} 74.1982  & \hspace{1mm} 9.7618  & 0.9608  & \hspace{1mm} 92.1311  & 12.1211 & 0.9701  & \hspace{1mm} 73.8442  & \hspace{1mm} 9.7152 \\
\hspace{3mm} 2500 & 0.9679  & \hspace{1mm} 94.6896  & 12.1024 & 0.9633  & 106.1061 & 13.5615 & 0.9692  & \hspace{1mm} 90.5637  & 11.5750 \\
\hspace{3mm} 3000 & 0.9708  & 101.9465 & 12.7332 & 0.9655  & 117.9213 & 14.7284 & 0.9720  & \hspace{1mm} 98.1398  & 12.2577 \\
\hspace{3mm} 3500 & 0.9711  & 115.4750 & 14.1504 & 0.9666  & 130.7154 & 16.0180 & 0.9704  & 117.2357 & 14.3662 \\
\hline
\end{tabular}%
}
\end{table*}


We further extend our simulations to a multi-treatment scenario involving three competing treatments, as motivated by the theoretical framework in Section \ref{Section 5}. In this case, we focus on the second-largest allocation count, as the maximal is invariably associated with the best-performing treatment under the allocation rule. The corresponding results, shown in \autoref{Table 4}, again indicate that this second-largest allocation remains finite as $N$ increases, consistent with our theoretical predictions.

For comparative purposes, we provide simulation results in \autoref{Table 5} and \autoref{Table 6} based on \emph{Procedure III} of Bhandari et al. \citeyearpar{Bhandari2009}. These demonstrate a key difference between our method and their earlier proposal: in their procedure, $\mathbb{E}(N^\prime_{1, N})/\log (N)$ stabilizes to a constant, whereas our approach maintains that $\mathbb{E}(N_{1,N})$ stays uniformly bounded across increasing $N$. This highlights the efficiency of our adaptive allocation strategy in the long-run performance.

Finally, we note that PCS is estimated as the mean of $10,000$ independent Bernoulli trials, each taking values $0$ or $1$. Hence, by basic binomial variance arguments, the standard error of the PCS estimate is bounded above by $\sqrt{\frac{1}{4 \times 10000}} = 0.005$. This ensures high precision of the simulation estimates and justifies the reliability of our numerical conclusions.

\newpage

\begin{remark}
    Note that, to compare our proposed two-treatment procedure with \emph{Procedure III} of Bhandari et al. \citeyearpar{Bhandari2009}, identical parameter settings were used for both the Normal and Bernoulli simulations. A comparison between Tables \ref{Table 1} and \ref{Table 2} (representing our method) and Tables \ref{Table 5} and \ref{Table 6} (representing \emph{Procedure III}) yields that our analyses are far more effective under the respective parameter configurations.
\end{remark}

\clearpage

\newpage

\section{Real Data Analysis} \label{Section 7}
\setlength{\parindent}{20pt}

We consider the following datasets to illustrate the practical relevance and effectiveness of our proposed method.

\subsection{Pregabalin Drug Trial for Postherpetic Neuralgia (PHN)} \label{Pregabalin drug data set}

We analyze the performance of our proposed two-treatment adaptive allocation strategy, using data from a randomized, double-blind, placebo-controlled trial to evaluate the efficacy of the drug Pregabalin for treating postherpetic neuralgia (PHN), as described in Dworkin et al. \citeyearpar{Pregabalin} and revisited in Biswas et al. \citeyearpar{PhaseIII2008}, pp. $47$.

In the original trial, patients were randomly assigned to either Pregabalin or placebo arms, and pain intensity was measured using an $11$-point numerical rating scale over an $8$-week period, where lower scores indicate better outcomes. The endpoint mean pain scores were $3.60$ (Pregabalin) and $5.29$ (placebo), with corresponding standard deviations $2.25$ and $2.20$, respectively.

As we do not have access to the raw dataset, we instead reconstruct a typical example of the data structure using these summary statistics. To adapt this continuous response setting to our model-based framework in Sections \ref{Section 2} -- \ref{Section 4}, we modeled the responses for Pregabalin and placebo groups as independent Normal distributions, using the reported endpoint mean scores and standard deviations. As lower mean pain scores indicate greater efficacy, the negation of the original mean scores reflects the assumption that higher values are favorable in our framework. Specifically, we use the following values as the true parameters:
\begin{itemize}
    \item Pregabalin (treatment $0$): mean $\theta_0 = -3.60$, sd $\sigma_0 = 2.25$,
    \item Placebo (treatment $1$): mean $\theta_1 = -5.29$, sd $\sigma_1 = 2.20$.
\end{itemize}

Using these parameters, we redesign sequential trials and generated $10,000$ replications following the two-treatment adaptive allocation rule for varying total sample sizes $N$ and computed the PCS as well as ($\mathbb{E}(N_{1,N})$), as considered in Section \ref{Section 6}. The results are presented in \autoref{Table 7}.

\newpage

\begin{table*}[h!]
\centering
\caption{Results for Pregabalin trial dataset.} \label{Table 7}
\begin{tabular}{c c c}
\hline
Total Sample Size $(N)$ 
& \multicolumn{1}{c}{PCS} 
& \multicolumn{1}{c}{$\mathbb{E}(N_{1, N})$} \\
\hline
\hspace{1mm} 200  & 0.9381 & 7.4673 \\
\hspace{1mm} 300  & 0.9346 & 7.6999 \\
\hspace{1mm} 400  & 0.9346 & 7.7434 \\
\hspace{1mm} 800  & 0.9326 & 8.5777 \\
\hspace{1mm} 900  & 0.9380 & 8.4504 \\
1000 & 0.9372 & 8.3561 \\
1500 & 0.9362 & 8.5037 \\
2000 & 0.9376 & 9.2875 \\
\hline
\end{tabular}
\end{table*}

The adaptive rule showed a marked preference for allocating patients to the superior treatment (Pregabalin), resulting in a high PCS (approximately $0.94$) and reduced average sample size for the inferior treatment (placebo) i.e., consistent with our theoretical results in \autoref{Theorem 3.1}. These calculations show a significant reduction in sample size allocated to the less effective treatment.



\subsection{Fluoxetine drug trial for Depressive Disorder} \label{Fluoxetine drug data set}

To evaluate the adaptive performance in real clinical settings with binary response, we examined a well-documented adaptive clinical trial conducted by Tamura et al. \citeyearpar{Fluoxetine} on the antidepressant drug Fluoxetine, compared to a placebo, which employed an adaptive design using a randomized play-the-winner rule (RPTW) and further discussed in Biswas et al. \citeyearpar{PhaseIII2008}, pp. $46$.

In this trial, the binary response variable was defined by the Hamilton Depression Rating Scale (\textrm{HAMD\textsubscript{17}}): patients achieving at least a $50\%$ reduction in \textrm{HAMD\textsubscript{17}} score after a minimum of $3$-weeks of therapy were labeled as responders. We apply the adaptive allocation methodology specifically to those patients assigned to the shortened REML stratum as considered in Biswas et al. \citeyearpar{PhaseIII2008}. To fit this into our framework introduced in Section \ref{Section 2}, we treat the response outcomes as arising from Bernoulli distributions for both treatments --- Fluoxetine and placebo. Specifically, we consider the observed sample proportions of responders for Fluoxetine and placebo as the true success probabilities $p_0$ and $p_1$, respectively. The reported response rates were:
\begin{itemize}
    \item Fluoxetine (treatment $0$): $p_0 = 0.58$,
    \item Placebo (treatment $1$): $p_1 = 0.36$.
\end{itemize}

Using these parameters, we construct a representative data structure which is used to redesign a series of sequential trials under our proposed adaptive allocation rule and stopping conditions described in Sections \ref{Section 2} -- \ref{Section 4}. We then carry out the trial for $10,000$ independent replications using $p_0 = 0.58$ and $p_1 = 0.36$ as the true success probabilities. \autoref{Table 8} summarizes the  results.

\begin{table*}[h!]
\centering
\caption{Results for Fluoxetine trial dataset.} \label{Table 8}
\begin{tabular}{c c c}
\hline
Total Sample Size $(N)$ 
& \multicolumn{1}{c}{PCS} 
& \multicolumn{1}{c}{$\mathbb{E}(N_{1, N})$} \\
\hline
\hspace{1mm} 200  & 0.9419 & 21.5476 \\
\hspace{1mm} 300  & 0.9373 & 22.1276 \\
\hspace{1mm} 400  & 0.9412 & 22.4579 \\
\hspace{1mm} 800  & 0.9424 & 22.8271 \\
\hspace{1mm} 900  & 0.9384 & 22.9035 \\
1000 & 0.9415 & 23.1794 \\
1500 & 0.9411 & 23.0954 \\
2000 & 0.9407 & 23.5513 \\
\hline
\end{tabular}
\end{table*}

Across increasing total sample sizes $N$, we observe a consistently high PCS (approximately $0.94$), affirming the design’s ability to identify the superior treatment. Notably, the number of patients allocated to the less effective treatment (placebo) --- remained stably low, supporting our theoretical result from Section \ref{Section 3}. Similar to the conclusion drawn from the data analysis in Section \ref{Pregabalin drug data set}, the proposed method may also have contributed to a reduction in sample size allocated to the less effective treatment in this case also.

\begin{remark}
    These analyses exemplify the broad applicability and effectiveness of the proposed adaptive allocation rules under realistic clinical contexts where treatment effects are significantly different --- as seen in the case of continuous responses for Pregabalin's efficacy in managing PHN-related pain and binary responses in the Fluoxetine drug trial for depressive disorder. In contrast to prior findings by Bhandari et al. \citeyearpar{Bhandari2009}, where the allocation count to the less effective treatment was shown to grow logarithmically with sample size under a large initial sample regime, our analyses show boundedness of $\mathbb{E}(N_{1,N})$ without requiring large pilot samples or complex tuning --- thus confirming the practical and ethical efficiency of our framework. Both analyses emphasize the ethical advantage of our design: minimizing patient exposure to inferior treatments while maintaining statistical efficiency. This supports the effectiveness of our allocation strategy under realistic conditions in general, and in particular for medical studies, as shown in the datasets considered.
\end{remark}



\section{Conclusion} \label{Section 8}
\setlength{\parindent}{20pt}

In this study, we considered and refined a known adaptive sequential sampling procedure in the context of two-treatment clinical trials, showing that the number of applications of the less effective drug remains finite with all moments bounded. Our theoretical results leverage Chernoff's theorem, which extends applicability beyond normal responses to a broader range of response distributions, and hence it has meaningful practical implications. We further generalized the method to a multi-treatment setup and established similar finiteness properties for allocations to suboptimal treatments.

Comprehensive simulation studies verified that the empirical average of allocations to the less effective treatment remains low and stable across increasing sample sizes. This contrasts with earlier findings of logarithmic growth, and highlights that boundedness can be achieved without large pilot samples or complex tuning. Real-data analyses further supported the practical viability of our method, demonstrating high PCS and minimized exposure to inferior treatments. Overall, the proposed adaptive strategy thus provides a statistically efficient and ethically sound framework for sequential decision-making in two or multi-armed experimental settings (clinical trials).


\bibliography{references}

\end{document}